\documentclass[12pt]{article}
\usepackage[utf8]{inputenc}
\usepackage{geometry}
\usepackage{amssymb}
\usepackage{amsthm}
\usepackage{amsmath}
\usepackage{enumerate}
\usepackage{booktabs}
\usepackage{graphicx}
\usepackage{multirow}
\usepackage{mathtools}
\usepackage{color}
\usepackage{algorithm,algorithmic}
\usepackage{authblk}

\newtheorem{theorem}{Theorem}
\newtheorem{corollary}{Corollary}
\newtheorem{lemma}{Lemma}
\newtheorem{definition}{Definition}
\newtheorem{property}{Property}
\newtheorem{remark}{Remark}

\newcommand{\R}{\mathbb{R}}

\title{Unit triangular factorization of the matrix symplectic group}
\author[1,2]{Pengzhan Jin}
\author[1,2]{Yifa Tang}
\author[1,2]{Aiqing Zhu}
\affil[1]{LSEC, ICMSEC, Academy of Mathematics and Systems Science, Chinese Academy of Sciences, Beijing 100190, China}
\affil[2]{School of Mathematical Sciences, University of Chinese Academy of Sciences, Beijing 100049, China}
\date{}

\begin{document}

\maketitle

\begin{abstract}
In this work, we prove that any symplectic matrix can be factored into no more than 9 unit triangular symplectic matrices. This structure-preserving factorization of the symplectic matrices immediately reveals two well-known features that, (i) the determinant of any symplectic matrix is one, (ii) the matrix symplectic group is path connected, as well as a new feature that (iii) all the unit triangular symplectic matrices form a set of generators of the matrix symplectic group. Furthermore, this factorization yields effective methods for the unconstrained parametrization of the matrix symplectic group as well as its structured subsets. The unconstrained parametrization enables us to apply faster and more efficient unconstrained optimization algorithms to the problems with symplectic constraints under certain circumstances.
\end{abstract}

\begin{keywords}
symplectic, symplectic matrix, triangular factorization, positive definite, singular symplectic matrix, generators, parametrization, unconstrained parametrization
\end{keywords}

\section{Introduction}
Denote the $d$-by-$d$ identity matrix by $I_{d}$, let
\begin{equation*}
J\coloneqq\begin{bmatrix} 0 & I_{d} \\ -I_{d} & 0 \end{bmatrix},
\end{equation*}
which is an orthogonal, skew-symmetric real matrix, so that $J^{-1}=J^{T}=-J$.
\begin{definition} \label{def:SP}
A matrix $H\in\R^{2d\times 2d}$ is called symplectic if $H^{T}JH=J$.
\end{definition}
We denote the collection of symplectic matrices by
\begin{equation*}
SP=\{H\in\R^{2d\times 2d}|H^{T}JH=J\},
\end{equation*}
which forms a group, i.e., matrix symplectic group. The group $SP$ is relevant from both the pure mathematical point of view \cite{fomenko1995symplectic}, and the point of view of applications. For instance, symplectic matrices play an important role in classical mechanics
and Hamiltonian dynamical systems \cite{abraham1978foundations,arnol2013mathematical,arnold2007mathematical}. The symplectic matrices also arise in the symplectic integrators which are the preferred methods for the numerical solutions of the differential equations appearing in these physical problems preserving the Hamiltonian structure \cite{feng2010symplectic,hairer2006geometric,sanz2018numerical}. Moreover, it appears in linear control \cite{fassbender2007symplectic,mehrmann1991autonomous}, optimal control \cite{hassibi1999indefinite}, the theory of parametric resonance \cite{iakubovich1975linear}. The applications have motivated the development of computing certain specific factorizations of symplectic matrices \cite{benzi2007iwasawa}.

There have been several modern factorizations of symplectic matrices. \cite{mehrmann1988symplectic} provides the block LDU factorization of symplectic matrices which require the nonsingular left upper block. \cite{dopico2006complementary,dopico2009parametrization} extend the LDU factorization to adapt general symplectic matrices by symplectic interchanges and permutation matrices. There are also QR-like factorization \cite{bunse1986matrix}, polar factorization \cite{higham2006symmetric}, as well as symplectic SVD factorization \cite{xu2003svd}. In addition, \cite{flaschka1991analysis,mackey2003determinant} express the $2d$-by-$2d$ symplectic matrices by the products of at most $4d$ symplectic transvections.

As aforementioned, we briefly present these modern factorizations as follows.
\begin{itemize}
\item
\textbf{Block LDU parametrization.} Before describing the block LDU parametrization, we first show the definition of symplectic interchange.
\begin{definition}
Let $1\leq j\leq d$. The symplectic interchange matrix $\Pi_{j}$ is the $2d$-by-$2d$ matrix obtained by interchanging the columns $j$ and $j+d$ of the $2d$-by-$2d$ identity matrix and multiplying the $j$-th column of the resulting matrix by $-1$. The symplectic interchange matrix $\widetilde{\Pi}_{j}$ is the $2d$-by-$2d$ matrix obtained by interchanging the columns $j$ and $j+d$ of the $2d$-by-$2d$ identity matrix and multiplying the $j+d$-th column of the resulting matrix by $-1$. Notice that $\Pi_{j}^{T}=\widetilde{\Pi}_{j}$.
\end{definition}
With the above definition, there is the block LDU parametrization.
\begin{theorem}
The set of $2d$-by-$2d$ symplectic matrices is
\begin{equation*}
\begin{split}
SP=&\Bigg\{Q\begin{bmatrix} I & 0 \\ C & I \end{bmatrix}\begin{bmatrix} G & 0 \\ 0 & G^{-T} \end{bmatrix}\begin{bmatrix} I & E \\ 0 & I \end{bmatrix}\Bigg|G\in\R^{d\times d}\ nonsingular, \\
&C=C^{T},\ E=E^{T},\ Q\ a\ product\ of\ symplectic\ interchanges\Bigg\}.
\end{split}
\end{equation*}
\end{theorem}
The detailed information of the block LDU parametrization can be found in \cite{dopico2009parametrization}.
\item
\textbf{QR-like factorization.}
\begin{theorem}
The set of $2d$-by-$2d$ symplectic matrices is
\begin{equation*}
SP=\left\{Q\begin{bmatrix} R & Z \\ 0 & R^{-T} \end{bmatrix}\Bigg|\begin{aligned}&R \in \R^{d\times d}\ upper\ triangular\\ &Q\ symplectic\ orthogonal\end{aligned}\right\}.
\end{equation*}
\end{theorem}
The detailed information of the QR-like factorization can be found in \cite{bunse1986matrix}.
\item
\textbf{Polar factorization.}
\begin{theorem}
The set of $2d$-by-$2d$ symplectic matrices is
\begin{equation*}
SP=\left\{QP\Bigg|\begin{aligned}&P=P^{T}\ symplectic\ positive\ definite\\ &Q\ symplectic\ orthogonal\end{aligned}\right\}.
\end{equation*}
\end{theorem}
The detailed information of the polar factorization can be found in \cite{higham2006symmetric}.
\item
\textbf{SVD-like factorization.}
\begin{theorem}
The set of $2d$-by-$2d$ symplectic matrices is
\begin{equation*}
SP=\left\{U\begin{bmatrix} \Omega & 0 \\ 0 & \Omega^{-1} \end{bmatrix}V^{T}\Bigg|\begin{aligned}&U,V\ symplectic\ orthogonal\\ &\Omega = diag(\omega_{1},\cdot\cdot\cdot,\omega_{d})\\ &\omega_{1}\geq\cdot\cdot\cdot\geq\omega_{d}\geq 1\end{aligned}\right\}.
\end{equation*}
\end{theorem}
The detailed information of the SVD-like factorization can be found in \cite{xu2003svd}.
\item
\textbf{Transvections factorization.} In text books on classical groups like \cite{artin2016geometric,taylor1992geometry,weyl1946classical} this result is mostly stated as a corollary of another basic fact, namely that the symplectic group is generated by so-called symplectic transvections.
\begin{definition}
For $0\neq u\in\R^{2d}$ and $0\neq\beta\in\R$, the matrix
\begin{equation*}
G=I+\beta uu^{T}J\in\R^{2d\times 2d}
\end{equation*}
is symplectic, and $G$ is called a symplectic transvection.
\end{definition}
Any symplectic matrix can be factored as the product of many symplectic transvections.
\begin{theorem}
The set of $2d$-by-$2d$ symplectic matrices is
\begin{equation*}
SP=\left\{G_{1}G_{2}\cdots G_{m}\Bigg|\begin{aligned}&G_{i}\ a\ symplectic\ transvection\\ &m\leq 4d\end{aligned}\right\}.
\end{equation*}
\end{theorem}
Although this factorization can freely parameterize the symplectic matrices, in numerical practice, the computed product of lots of symplectic matrices may be far from being symplectic when $d$ is large. There are some works parameterizing the symplectic matrices as the finite product of certain elementary units \cite{kronsbein1964note,murnaghan1953parametrisation}, however, which take effect on only special structured symplectic matrices. The detailed information of the transvections factorization can be found in \cite{flaschka1991analysis,mackey2003determinant}.
\end{itemize}

In this work, we will show a more elementary factorization of the symplectic matrices. Considering
\begin{equation} \label{eq:tri_units}
\begin{bmatrix} I & S \\ 0 & I\end{bmatrix},\quad \begin{bmatrix} I & 0 \\ S & I\end{bmatrix},
\end{equation}
 where $S$ is symmetric, we will factor the symplectic matrices as the products of at most 9 \textbf{unit triangular symplectic matrices} shown above. In addition, we call the symplectic matrices like
 \begin{equation} \label{eq:dia_units}
 \begin{bmatrix} P & 0 \\ 0 & P^{-T} \end{bmatrix}
 \end{equation}
as \textbf{diagonal symplectic matrices}. These two types of symplectic matrices are concerned as the basic units in our factorizations.

Since the symplectic matrix with nonsingular left upper block is easy to deal with, in classical factorizations, researchers have to exploit some intractable factors such as symplectic interchanges and permutation matrices to make a general symplectic matrix equipped with nonsingular left upper block. Different from the conventional methods, we find a family of unit triangular symplectic matrices which can transform a general symplectic matrix into one having nonsingular left upper block, it is the core of the theoretical results in this work. With this result, we are able to factor the symplectic matrix into three unit triangular symplectic matrices and one diagonal symplectic matrix, called \textbf{unit ULU factorization}. Furthermore, we decompose the diagonal symplectic matrix into several unit triangular symplectic matrices and subsequently obtain the final \textbf{unit triangular factorization}.

The unit triangular factorization immediately leads to many significant properties of the matrix symplectic group, such as, (\romannumeral1) the determinant of symplectic matrix is one, (\romannumeral2) the matrix symplectic group is path connected, (\romannumeral3) all the unit triangular symplectic matrices form a set of generators of the matrix symplectic group. In addition, we will consider about some structured subsets of the matrix symplectic group such as the symplectic matrices with nonsingular left upper block, the positive definite symplectic matrices and the symplectic M-matrices, which can be factored into less than 9 unit triangular symplectic matrices due to their special structures.

A remarkable characteristic of the unit triangular factorization is that it provides an unconstrained parametrization for symplectic matrices. For the general symplectic matrices, the positive definite symplectic matrices and the singular symplectic matrices, we will explore their unconstrained parametrization respectively. With the unconstrained parametrization, one may be able to apply faster and more efficient unconstrained optimization algorithms to the problems with symplectic constraints under certain circumstances, which will be discussed in detail later. It was in fact our work on optimization problem with symplectic constraints that led us to consider the unit triangular factorization of the symplectic matrices.

The remaining parts of this paper are organized as follows. Some basic properties that will be used later are given in Section \ref{sec:preliminaries}. Section \ref{sec:factorization} presents the detailed proofs of the factorizations and deduces several important inferences. In Section \ref{sec:parametrization}, we provide the unconstrained parametrization for the matrix symplectic group and its structured subsets. Some conclusions will be given in the last section.

\section{Preliminaries} \label{sec:preliminaries}
At first, we introduce the notations will be used as well as some relevant basic results for factorization of symplectic matrices. As shown in Definition \ref{def:SP}, in this work we consider the real $2d$-by-$2d$ matrix symplectic group $SP$. Property \ref{pro:smat_pro}-\ref{pro:smat_cond} are well-known and can be easily verified according to the definition of $SP$. We omit the proofs and the readers are referred to \cite{de2006symplectic,feng2010symplectic} for more details. Since the features of symplectic matrices are not apparent in Definition \ref{def:SP}, there is an equivalent condition of being symplectic.
\begin{property} \label{pro:smat_pro}
If $H=\begin{bmatrix} A_{1} & B_{1} \\ A_{2} & B_{2} \end{bmatrix}\in \R^{2d\times 2d}$ is a symplectic matrix and $A_{i},B_{i}\in \R^{d\times d}$, then
\begin{enumerate}[(i)]
    \item $A_{1}^{T}A_{2}=A_{2}^{T}A_{1}$,
    \item $B_{1}^{T}B_{2}=B_{2}^{T}B_{1}$,
    \item $A_{1}^{T}B_{2}-A_{2}^{T}B_{1}=I$,
\end{enumerate}
vice versa. Note that an $A=\begin{bmatrix}A_{1} \\ A_{2}\end{bmatrix}\in\R^{2d\times d}$ satisfying the first item is called a symmetric pair.
\end{property}
Property \ref{pro:smat_pro} points out the detailed relationship among the four blocks in a symplectic matrix, it plays an important role on unpacking the features. Based on Property \ref{pro:smat_pro}, one may check that a symplectic matrix being unit and triangular must be written in form (\ref{eq:tri_units}). Besides the inner structure of symplectic matrix, there also hold several overall natures, for instance, $SP$ remains closed under multiplication, transposition and inversion.
\begin{property} \label{pro:smat_trans}
If $H,G\in SP$, then
\begin{enumerate}[(i)]
    \item $HG\in SP$,
    \item $H^{T}\in SP$,
    \item $H^{-1}\in SP$.
\end{enumerate}
\end{property}
With the closure under multiplication and inversion, we can easily check that $SP$ immediately forms a group. In addition, the second item in Property \ref{pro:smat_trans} indicates that the equivalent condition in Property \ref{pro:smat_pro} can be written in another way, as
\begin{equation}
\left\{
\begin{aligned}
&A_{1}B_{1}^{T}=B_{1}A_{1}^{T} \\
&A_{2}B_{2}^{T}=B_{2}A_{2}^{T} \\
&A_{1}B_{2}^{T}-B_{1}A_{2}^{T}=I
\end{aligned}
\right..
\end{equation}

In this work, we are aiming to seek elementary factorizations based on only (\ref{eq:tri_units}) and (\ref{eq:dia_units}), which are the most concise symplectic matrices. To figure out how the unit triangular matrices and the diagonal matrices are being symplectic, here shows a property as follows.
\begin{property} \label{pro:smat_cond}
If $S,P,Q,A,B,C\in\R^{d\times d}$, then
\begin{enumerate}[(i)]
    \item The matrix $\begin{bmatrix} I & S \\ 0 & I \end{bmatrix}$ is symplectic if and only if $S^{T}=S$,
    \item The matrix $\begin{bmatrix} I & 0 \\ S & I \end{bmatrix}$ is symplectic if and only if $S^{T}=S$,
    \item The matrix $\begin{bmatrix} P & 0 \\ 0 & Q \end{bmatrix}$ is symplectic if and only if $Q=P^{-T}$.
    %\item The matrix $\begin{bmatrix} I & A \\ B & C \end{bmatrix}$ is symplectic if and only if $A=A^{T}$, $B=B^{T}$, $C=I+BA$.
\end{enumerate}
\end{property}
The mentioned two types of matrices, i.e., the unit triangular symplectic matrices and the diagonal symplectic matrices, have satisfactory properties and clear structures. In the later sections we will think of a way to do the unconstrained parametrization of symplectic matrices and its structured subsets, for this purpose, the unit triangular symplectic matrices (\ref{eq:tri_units}) are more preferred, since each one is determined by $\frac{d(d+1)}{2}$ free parameters while the diagonal symplectic matrices (\ref{eq:dia_units}) require the nonsingularity on $P$.

Although the determinant of symplectic matrix is obviously $\pm 1$, it is not easy to confirm that the determinant is always $1$. However, in some special cases, we can prove it without any difficulty. For instance, when the symplectic matrix has a nonsingular left upper block, we immediately derive such a LDU factorization whose factors are simple enough to obtain the determinants.
\begin{property}[LDU Factorization] \label{pro:LDU}
If $H=\begin{bmatrix} A_{1} & B_{1} \\ A_{2} & B_{2} \end{bmatrix}\in \R^{2d\times 2d}$ is a symplectic matrix and $A_{i},B_{i}\in \R^{d\times d}$, moreover $A_{1}$ is nonsingular, then $H$ has three unique factorizations
\begin{equation} \label{eq:left_LDU}
\left\{
\begin{aligned}
&H=\begin{bmatrix} P_{1} & 0 \\ 0 & P_{1}^{-T} \end{bmatrix}\begin{bmatrix} I & 0 \\ S_{1} & I \end{bmatrix}\begin{bmatrix} I & T_{1} \\ 0 & I \end{bmatrix} \\
&S_{1}=A_{1}^{T}A_{2},\ T_{1}=A_{1}^{-1}B_{1},\ P_{1}=A_{1} \\
\end{aligned}
\right.,
\end{equation}
\begin{equation} \label{eq:center_LDU}
\left\{
\begin{aligned}
&H=\begin{bmatrix} I & 0 \\ S_{2} & I \end{bmatrix}\begin{bmatrix} P_{2} & 0 \\ 0 & P_{2}^{-T} \end{bmatrix}\begin{bmatrix} I & T_{2} \\ 0 & I \end{bmatrix} \\
&S_{2}=A_{2}A_{1}^{-1},\ T_{2}=A_{1}^{-1}B_{1},\ P_{2}=A_{1} \\
\end{aligned}
\right.,
\end{equation}
\begin{equation} \label{eq:right_LDU}
\left\{
\begin{aligned}
&H=\begin{bmatrix} I & 0 \\ S_{3} & I \end{bmatrix}\begin{bmatrix} I & T_{3} \\ 0 & I \end{bmatrix}\begin{bmatrix} P_{3} & 0 \\ 0 & P_{3}^{-T} \end{bmatrix} \\
&S_{3}=A_{2}A_{1}^{-1},\ T_{3}=B_{1}A_{1}^{T},\ P_{3}=A_{1} \\
\end{aligned}
\right.,
\end{equation}
where $S_{1},S_{2},S_{3},T_{1},T_{2},T_{3}$ are symmetric and $P_{1},P_{2},P_{3}$ are nonsingular.
\end{property}
\begin{proof}
Here we show the proof of (\ref{eq:center_LDU}), and the others are the same. Assume that $S_{2},T_{2},P_{2}$ are defined as in (\ref{eq:center_LDU}), then
\begin{equation*}
\begin{split}
&S_{2}^{T}=(A_{2}A_{1}^{-1})^{T}=A_{1}^{-T}A_{2}^{T}A_{1}A_{1}^{-1}=A_{1}^{-T}A_{1}^{T}A_{2}A_{1}^{-1}=A_{2}A_{1}^{-1}=S_{2}, \\
&T_{2}^{T}=(A_{1}^{-1}B_{1})^{T}=A_{1}^{-1}A_{1}B_{1}^{T}A_{1}^{-T}=A_{1}^{-1}B_{1}A_{1}^{T}A_{1}^{-T}=A_{1}^{-1}B_{1}=T_{2}, \\
\end{split}
\end{equation*}
which mean $S_{2},T_{2}$ are symmetric. Furthermore,
\begin{equation*}
\begin{split}
P_{2}^{-T}+S_{2}P_{2}T_{2}=A_{1}^{-T}+A_{2}A_{1}^{-1}A_{1}A_{1}^{-1}B_{1}&=A_{1}^{-T}(I+A_{1}^{T}A_{2}A_{1}^{-1}B_{1}) \\
&=A_{1}^{-T}(I+A_{2}^{T}A_{1}A_{1}^{-1}B_{1}) \\
&=A_{1}^{-T}A_{1}^{T}B_{2} \\
&=B_{2},
\end{split}
\end{equation*}
therefore
\begin{equation*}
\begin{bmatrix} I & 0 \\ S_{2} & I \end{bmatrix}\begin{bmatrix} P_{2} & 0 \\ 0 & P_{2}^{-T} \end{bmatrix}\begin{bmatrix} I & T_{2} \\ 0 & I \end{bmatrix}=\begin{bmatrix} P_{2} & P_{2}T_{2} \\ S_{2}P_{2} & P_{2}^{-T}+S_{2}P_{2}T_{2} \end{bmatrix}=\begin{bmatrix} A_{1} & B_{1} \\ A_{2} & B_{2} \end{bmatrix}=H.
\end{equation*}
The uniqueness is easy to check and we omit it.
\end{proof}
Property \ref{pro:LDU} decomposes the symplectic matrices with nonsingular left upper block into two unit triangular symplectic matrices and one diagonal symplectic matrices, where the diagonal one can be placed at an arbitrary position. What we expect is to derive a similar result for the general symplectic matrices.

In order to complete our proofs, we also have to give the other properties for checking if some columns are a part of a symplectic matrix. The relevant properties can be found in \cite{feng2010symplectic,freiling2002existence}.
\begin{property} \label{pro:part_symm}
If $A=\begin{bmatrix}A_{1} \\ A_{2}\end{bmatrix}\in\R^{2d\times k},1\leq k\leq d$, $A$ is full-rank, and $A_{1}^{T}A_{2}=A_{2}^{T}A_{1}$, then there exists a $A'\in \R^{2d\times(d-k)}$ such that $\begin{bmatrix}A & A'\end{bmatrix}$ is a full-rank symmetric pair.
\end{property}
\begin{proof}
It is shown as \cite[p. 147, Theorem 3.27]{feng2010symplectic}.
\end{proof}
\begin{property} \label{pro:part_SP}
If $A\in\R^{2d\times d}$ is a full-rank symmetric pair, then there exists a $B\in\R^{2d\times d}$ such that $\begin{bmatrix}A & B\end{bmatrix}\in SP$.
\end{property}
\begin{proof}
It is shown as \cite[p. 147, Theorem 3.28]{feng2010symplectic}.
\end{proof}

At the end of this section, we provide a table including the main notations in this paper. See Table \ref{tab:notations}.
\begin{table}[htbp]
    \centering
    \begin{tabular}{p{40pt}|p{300pt}}
        \toprule
        $\R^{2d\times 2d}$ & All the symplectic matrices involved are in the real space $\R^{2d\times 2d}$. \\
        \midrule
        $J$ & $\begin{bmatrix} 0 & I_{d} \\ -I_{d} & 0 \end{bmatrix}$. \\
        \midrule
        $SP$ & The collection of all the $2d$-by-$2d$ symplecic matrices. \\
        \midrule
        $SPN$ & The collection of all the $2d$-by-$2d$ symplecic matrices with nonsingular left upper block. \\
        \midrule
        $SPP$ & The collection of all the $2d$-by-$2d$ symmetric positive definite symplecic matrices. \\
        \midrule
        $SPM$ & The collection of all the $2d$-by-$2d$ symplecic M-matrices. The definition of M-matrices is shown in Definition \ref{def:M-matrices}. \\
        \midrule
        $SPS$ & The collection of all the $2d$-by-$2d$ singular symplecic matrices. The definition of singular symplecic matrices is shown in Definition \ref{def:singular_SP}.\\
        \midrule
        $\mathcal{L}_{n}$ & $\left\{\begin{bmatrix} I & 0/S_{n} \\ S_{n}/0 & I \end{bmatrix}\cdots\begin{bmatrix} I & 0 \\ S_{2} & I \end{bmatrix}\begin{bmatrix} I & S_{1} \\ 0 & I \end{bmatrix}\Bigg|S_{i}\in\R^{d\times d},S_{i}^{T}=S_{i}\right\}$. The unit upper triangular symplectic matrices and the unit lower triangular symplectic matrices appear alternately. \\
        \midrule
        $\mathcal{L}_{n}^{2}$ & $\{LL^T|L\in\mathcal{L}_{n}\}$. A subset of $SPP$. \\
        \midrule
        $\mathcal{L}_{1}^{T}$ & $\{L^{T}|L\in\mathcal{L}_{1}\}$. The collection of all the unit lower triangular symplectic matrices. \\
        \midrule
        $SG$ & $\mathcal{L}_{1}\cup\mathcal{L}_{1}^{T}$. A set of generators of the group $SP$. \\
        \midrule
        $Pa$ & The map for extracting the lower triangular parameters for symmetric matrices. $Pa(S)=(s_{11},s_{21},s_{22},s_{31},\cdots,s_{dd})$, where $S=(s_{ij})\in \R^{d\times d}$, $S^{T}=S$. \\
        \bottomrule
    \end{tabular}
    \caption{The main notations in this paper.}
    \label{tab:notations}
\end{table}

\section{Unit triangular factorization} \label{sec:factorization}
In this section, we will present the detailed proof of the unit triangular factorization. To describe the problem clearly, denote
\begin{equation*}
    \mathcal{L}_{n}=\left\{\begin{bmatrix} I & 0/S_{n} \\ S_{n}/0 & I \end{bmatrix}\cdots\begin{bmatrix} I & 0 \\ S_{2} & I \end{bmatrix}\begin{bmatrix} I & S_{1} \\ 0 & I \end{bmatrix}\Bigg|S_{i}\in\R^{d\times d},S_{i}^{T}=S_{i},i=1,2,\cdots,n\right\},
\end{equation*}
where the unit upper triangular symplectic matrices and the unit lower triangular symplectic matrices appear alternately. And it is clear that $\mathcal{L}_{m}\subset \mathcal{L}_{n}\subset SP$ for all integers $1\leq m\leq n$. Now the main theorem is given as follows.
\begin{theorem}[Unit Triangular Factorization] \label{thm:fac_thm}
$SP=\mathcal{L}_{9}$. Thus any symplectic matrix can be factored into no more than $9$ unit triangular symplectic matrices.
\end{theorem}
To prove this theorem, we first list some necessary auxiliary results. For convenience, we denote all the unimportant blocks in matrices by ``$\star$'' throughout this paper.
\begin{theorem} \label{thm:nonsin_thm}
For any symplectic matrix $H\in\R^{2d\times 2d}$, there exists a symmetric $S\in \R^{d\times d}$ such that, the factorization
\begin{equation*}
    H=\begin{bmatrix} I & \lambda S \\ 0 & I \end{bmatrix}\begin{bmatrix} P_{\lambda} & \star \\ \star & \star \end{bmatrix}
\end{equation*}
holds with a nonsingular $P_{\lambda}$ for all $\lambda\neq 0$. Hence any symplectic matrix can be decomposed into a unit upper triangular symplectic matrix and a symplectic matrix with nonsingular left upper block. Furthermore, if needed, $S$ can be set to
\begin{equation*}
S=-P\begin{bmatrix} O_{r} & 0 \\ 0 & I_{d-r} \end{bmatrix}P^{T}
\end{equation*}
when the left upper block $A_{1}$ of $H$ with $rank\ r$ is decomposed as
\begin{equation*}
A_{1}=P\begin{bmatrix} I_{r} & 0 \\ 0 & O_{d-r} \end{bmatrix}Q
\end{equation*}
where $P,Q\in \R^{d\times d}$ are nonsingular.
\end{theorem}
\begin{proof}
Denote $H=\begin{bmatrix} A_{1} & B_{1} \\ A_{2} & B_{2} \end{bmatrix}\in SP$ where $A_{i},B_{i}\in \R^{d\times d}$, and assume that $rank(A_{1})=r$. There exist nonsingular $P,Q$ such that $P^{-1}A_{1}Q^{-1}=\begin{bmatrix} I_{r} & 0 \\ 0 & O_{d-r} \end{bmatrix}$, and let $S=-P\begin{bmatrix} O_{r} & 0 \\ 0 & I_{d-r} \end{bmatrix}P^{T}$, $\lambda\neq 0$. Then
\begin{equation*}
\begin{split}
&\begin{bmatrix} I & -\lambda P^{-1}SP^{-T} \\ 0 & I \end{bmatrix}\begin{bmatrix} P^{-1} & 0 \\ 0 & P^{T} \end{bmatrix}\begin{bmatrix} A_{1} & B_{1} \\ A_{2} & B_{2} \end{bmatrix}\begin{bmatrix} Q^{-1} & 0 \\ 0 & Q^{T} \end{bmatrix} \\
=&\begin{bmatrix} I & \lambda\begin{bmatrix} O_{r} & 0 \\ 0 & I_{d-r} \end{bmatrix} \\ 0 & I \end{bmatrix}\begin{bmatrix} \begin{bmatrix} I_{r} & 0 \\ 0 & O_{d-r} \end{bmatrix} & \star \\ \begin{bmatrix} C_{1} & C_{2} \\ C_{3} & C_{4} \end{bmatrix} & \star \end{bmatrix} \\
=&\begin{bmatrix} \begin{bmatrix} I_{r} & 0 \\ \lambda C_{3} & \lambda C_{4} \end{bmatrix} & \star \\ \star & \star \end{bmatrix} \\
=&\begin{bmatrix} D_{\lambda} & \star \\ \star & \star \end{bmatrix}
\end{split}
\end{equation*}
where $C_{1}\in\R^{r\times r},C_{2}\in\R^{r\times(d-r)},C_{3}\in\R^{(d-r)\times r},C_{4}\in\R^{(d-r)\times(d-r)}$ as well as $D_{\lambda}=\begin{bmatrix} I_{r} & 0 \\ \lambda C_{3} & \lambda C_{4} \end{bmatrix}$, next we explain why $D_{\lambda}$ is nonsingular. Since all the matrices involved above are symplectic, we know that $\begin{bmatrix} I_{r} & 0 \\ 0 & O_{d-r} \end{bmatrix}$ and $\begin{bmatrix} C_{1} & C_{2} \\ C_{3} & C_{4} \end{bmatrix}$ satisfy
\begin{equation*}
\begin{bmatrix} I_{r} & 0 \\ 0 & O_{d-r} \end{bmatrix}^{T}\begin{bmatrix} C_{1} & C_{2} \\ C_{3} & C_{4} \end{bmatrix}=\begin{bmatrix} C_{1} & C_{2} \\ C_{3} & C_{4} \end{bmatrix}^{T}\begin{bmatrix} I_{r} & 0 \\ 0 & O_{d-r} \end{bmatrix}
\end{equation*}
due to Property \ref{pro:smat_pro}, which implies that $C_{2}=0$. Furthermore,
\begin{equation*}
\begin{bmatrix} I_{r} & 0 \\ 0 & O_{d-r} \\ C_{1} & C_{2} \\ C_{3} & C_{4} \end{bmatrix}=\begin{bmatrix} I_{r} & 0 \\ 0 & 0 \\ C_{1} & 0 \\ C_{3} & C_{4} \end{bmatrix}
\end{equation*}
is full-rank as the sub-columns of a symplectic matrix, thus $C_{4}$ is also full-rank, i.e., nonsingular. Hence $D_{\lambda}=\begin{bmatrix} I_{r} & 0 \\ \lambda C_{3} & \lambda C_{4} \end{bmatrix}$ is nonsingular.

Now we have
\begin{equation*}
\begin{split}
H=\begin{bmatrix} A_{1} & B_{1} \\ A_{2} & B_{2} \end{bmatrix}&=\begin{bmatrix} P & 0 \\ 0 & P^{-T} \end{bmatrix}\begin{bmatrix} I & \lambda P^{-1}SP^{-T} \\ 0 & I \end{bmatrix}\begin{bmatrix} D_{\lambda} & \star \\ \star & \star \end{bmatrix}\begin{bmatrix} Q & 0 \\ 0 & Q^{-T} \end{bmatrix} \\
&=\begin{bmatrix} P & \lambda SP^{-T} \\ 0 & P^{-T} \end{bmatrix}\begin{bmatrix} D_{\lambda} & \star \\ \star & \star \end{bmatrix}\begin{bmatrix} Q & 0 \\ 0 & Q^{-T} \end{bmatrix} \\
&=\begin{bmatrix} I & \lambda S \\ 0 & I \end{bmatrix}\begin{bmatrix} P & 0 \\ 0 & P^{-T} \end{bmatrix}\begin{bmatrix} D_{\lambda} & \star \\ \star & \star \end{bmatrix}\begin{bmatrix} Q & 0 \\ 0 & Q^{-T} \end{bmatrix} \\
&=\begin{bmatrix} I & \lambda S \\ 0 & I \end{bmatrix}\begin{bmatrix} PD_{\lambda}Q & \star \\ \star & \star \end{bmatrix} \\
&=\begin{bmatrix} I & \lambda S \\ 0 & I \end{bmatrix}\begin{bmatrix} P_{\lambda} & \star \\ \star & \star \end{bmatrix}
\end{split}
\end{equation*}
where $P_{\lambda}=PD_{\lambda}Q$ is nonsingular for all $\lambda\neq 0$.
\end{proof}
\begin{remark}
In the most cases, we are unconcerned about the $\lambda$ in Theorem \ref{thm:nonsin_thm}, and only focus on the factorization
\begin{equation*}
H=\begin{bmatrix} I & S \\ 0 & I \end{bmatrix}\begin{bmatrix} P & \star \\ \star & \star \end{bmatrix}
\end{equation*}
where $S$ is symmetric and $P$ is nonsingular. It in fact yields a symplectic matrix with nonsingular left upper block which is much more easier to deal with.
\end{remark}
To prove the determinant of any symplectic matrix is one, researchers make efforts to have the symplectic matrix equipped with nonsingular left upper block, however, all of their transformations require the intractable blocks like symplectic interchanges and permutation matrices. Theorem \ref{thm:nonsin_thm} is the core of the proof of Theorem \ref{thm:fac_thm}, it transforms a symplectic matrix into one having nonsingular left upper block by one unit triangular symplectic matrix.

\begin{lemma}[Unit ULU Factorization] \label{lem:unit_ULU}
For any symplectic matrix $H\in \R^{2d\times 2d}$, there exist symmetric $S,T,U\in \R^{d\times d}$ and a nonsingular $P\in \R^{d\times d}$ such that
\begin{equation*}
H=\begin{bmatrix} I & S \\ 0 & I \end{bmatrix}\begin{bmatrix}P & \star \\ \star & \star \end{bmatrix}=\begin{bmatrix} I & S \\ 0 & I \end{bmatrix}\begin{bmatrix} I & 0 \\ T & I \end{bmatrix}\begin{bmatrix} I & U \\ 0 & I \end{bmatrix}\begin{bmatrix} P & 0 \\ 0 & P^{-T} \end{bmatrix}.
\end{equation*}
Furthermore, $T,U,P$ are uniquely determined by $H$ and $S$.
\end{lemma}
\begin{proof}
Theorem \ref{thm:nonsin_thm} shows that there exists the factorization
\begin{equation*}
H=\begin{bmatrix} I & S \\ 0 & I \end{bmatrix}\begin{bmatrix} P & \star \\ \star & \star \end{bmatrix}
\end{equation*}
where $S$ is symmetric and $P$ is nonsingular. Moreover, by (\ref{eq:right_LDU}),
\begin{equation*}
\begin{bmatrix} P & \star \\ \star & \star \end{bmatrix}=\begin{bmatrix} P & A \\ B & C \end{bmatrix}=\begin{bmatrix} I & 0 \\ BP^{-1} & I \end{bmatrix}\begin{bmatrix} I & AP^{T} \\ 0 & I \end{bmatrix}\begin{bmatrix} P & 0 \\ 0 & P^{-T} \end{bmatrix},
\end{equation*}
hence
\begin{equation*}
H=\begin{bmatrix} I & S \\ 0 & I \end{bmatrix}\begin{bmatrix} I & 0 \\ T & I \end{bmatrix}\begin{bmatrix} I & U \\ 0 & I \end{bmatrix}\begin{bmatrix} P & 0 \\ 0 & P^{-T} \end{bmatrix}
\end{equation*}
where $T=BP^{-1}$ and $U=AP^{T}$ are symmetric. In addition, when $H$ and $S$ are fixed, we are able to obtain $P,A,B,C,T,U$ in sequence, thus the factorization is determined by $H,S$. Note again that all the matrices involved above are symplectic.

\end{proof}

\begin{remark} \label{rem:moving}
Notice that
\begin{equation*}
\begin{bmatrix} P & 0 \\ 0 & P^{-T} \end{bmatrix}\begin{bmatrix} I & S \\ 0 & I \end{bmatrix}=\begin{bmatrix} I & PSP^{T} \\ 0 & I \end{bmatrix}\begin{bmatrix} P & 0 \\ 0 & P^{-T} \end{bmatrix}.
\end{equation*}
The diagonal symplectic matrices like $\begin{bmatrix} P & 0 \\ 0 & P^{-T} \end{bmatrix}$ can be moved to any position between a series of many $\begin{bmatrix} I & \star \\ 0 & I \end{bmatrix}$ and $\begin{bmatrix} I & 0 \\ \star & I \end{bmatrix}$. For instance,
\begin{equation*}
\begin{split}
\begin{bmatrix} P & 0 \\ 0 & P^{-T} \end{bmatrix}\begin{bmatrix} I & 0 \\ \star & I \end{bmatrix}\begin{bmatrix} I & \star \\ 0 & I \end{bmatrix}&=\begin{bmatrix} I & 0 \\ \star & I \end{bmatrix}\begin{bmatrix} P & 0 \\ 0 & P^{-T} \end{bmatrix}\begin{bmatrix} I & \star \\ 0 & I \end{bmatrix} \\
&=\begin{bmatrix} I & 0 \\ \star & I \end{bmatrix}\begin{bmatrix} I & \star \\ 0 & I \end{bmatrix}\begin{bmatrix} P & 0 \\ 0 & P^{-T} \end{bmatrix},
\end{split}
\end{equation*}
therefore we can adjust the position of diagonal symplectic matrices among the unit triangular symplectic matrices by moving, even merge them into one.
\end{remark}

What we just did is applying the LDU factorization to Theorem \ref{thm:nonsin_thm} in the proof of Lemma \ref{lem:unit_ULU}. According to Remark \ref{rem:moving}, Lemma \ref{lem:unit_ULU} has some other forms, like
\begin{equation*}
H=\begin{bmatrix} I & S \\ 0 & I \end{bmatrix}\begin{bmatrix} I & 0 \\ \star & I \end{bmatrix}\begin{bmatrix} P & 0 \\ 0 & P^{-T} \end{bmatrix}\begin{bmatrix} I & \star \\ 0 & I \end{bmatrix}
\end{equation*}
and
\begin{equation*}
H=\begin{bmatrix} I & S \\ 0 & I \end{bmatrix}\begin{bmatrix} P & 0 \\ 0 & P^{-T} \end{bmatrix}\begin{bmatrix} I & 0 \\ \star & I \end{bmatrix}\begin{bmatrix} I & \star \\ 0 & I \end{bmatrix}.
\end{equation*}
The most important fact is the number of unit triangular matrices, rather than the position of diagonal matrix. Note that all the forms involved share the same $P$ due to Property \ref{pro:LDU}.

Actually, Theorem \ref{thm:nonsin_thm} and Lemma \ref{lem:unit_ULU} provide an algorithm for unit ULU factorization, see Algorithm \ref{alg:unit_ULU}.
\begin{algorithm}
\caption{Unit ULU factorization}
\label{alg:unit_ULU}
\begin{algorithmic}
\REQUIRE{$H\in SP$}
\ENSURE{$S,T,U,P$ such that $H=\begin{bmatrix} I & S \\ 0 & I \end{bmatrix}\begin{bmatrix} I & 0 \\ T & I \end{bmatrix}\begin{bmatrix} I & U \\ 0 & I \end{bmatrix}\begin{bmatrix} P & 0 \\ 0 & P^{-T} \end{bmatrix}$}
\STATE{Given $H:=\begin{bmatrix}A&\star \\ \star&\star\end{bmatrix}\in SP$}
\STATE{Compute the factorization $A=P\begin{bmatrix} I_{r} & 0 \\ 0 & O_{d-r} \end{bmatrix}Q$ where $P,Q$ are nonsingular}
\STATE{Set $S:=-P\begin{bmatrix} O_{r} & 0 \\ 0 & I_{d-r} \end{bmatrix}P^{T}$}
\STATE{Set $\begin{bmatrix} A_{1} & B_{1} \\ A_{2} & B_{2} \end{bmatrix}:=\begin{bmatrix} I & -S \\ 0 & I \end{bmatrix}H$}
\STATE{Set $T:=A_{2}A_{1}^{-1}$}
\STATE{Set $U:=B_{1}A_{1}^{T}$}
\STATE{Set $P:=A_{1}$}
\RETURN $S,T,U,P$
\end{algorithmic}
\end{algorithm}

So far we have proved that any symplectic matrix can be factored into three unit triangular symplectic matrices together with one diagonal symplectic matrix, where the number three is optimal, since the product of two only yields a symplectic matrix with nonsingular left upper block. To obtain the final unit triangular factorization, we have to consider about how to decompose the diagonal symplectic matrices next.
\begin{lemma} \label{lem:decomp}
For any symplectic matrix $H=\begin{bmatrix} P & 0 \\ 0 & P^{-T} \end{bmatrix}\in \R^{2d\times 2d}$ where $P\in \R^{d\times d}$ is nonsingular, there exist symmetric $S_{1},S_{2},\cdots,S_{7}\in \R^{d\times d}$ as well as symmetric $T_{1},T_{2},\cdots,T_{7}\in \R^{d\times d}$ such that
\begin{equation*}
\begin{split}
    H&=\begin{bmatrix} I & S_{7} \\ 0 & I \end{bmatrix}\begin{bmatrix} I & 0 \\ S_{6} & I \end{bmatrix}\cdots\begin{bmatrix} I & 0 \\ S_{2} & I \end{bmatrix}\begin{bmatrix} I & S_{1} \\ 0 & I \end{bmatrix} \\
    &=\begin{bmatrix} I & 0 \\ T_{7} & I \end{bmatrix}\begin{bmatrix} I & T_{6} \\ 0 & I \end{bmatrix}\cdots\begin{bmatrix} I & T_{2} \\ 0 & I \end{bmatrix}\begin{bmatrix} I & 0 \\ T_{1} & I \end{bmatrix}.
\end{split}
\end{equation*}
\end{lemma}
\begin{proof}
At first, there exist two symmetric matrices $P_{1},P_{2}\in \R^{d\times d}$ such that $P=P_{1}\cdot P_{2}$, then we have
\begin{equation*}
    H=\begin{bmatrix} P & 0 \\ 0 & P^{-T} \end{bmatrix}=\begin{bmatrix} P_{1} & 0 \\ 0 & P_{1}^{-T} \end{bmatrix}\begin{bmatrix} P_{2} & 0 \\ 0 & P_{2}^{-T} \end{bmatrix}.
\end{equation*}
Moreover,
\begin{equation*}
\begin{bmatrix} P_{1} & 0 \\ 0 & P_{1}^{-T} \end{bmatrix}=\begin{bmatrix} I & -P_{1} \\ 0 & I \end{bmatrix}\begin{bmatrix} I & 0 \\ P_{1}^{-1}-I & I \end{bmatrix}\begin{bmatrix} I & I \\ 0 & I \end{bmatrix}\begin{bmatrix} I & 0 \\ P_{1}-I & I \end{bmatrix},
\end{equation*}
\begin{equation*}
\begin{bmatrix} P_{2} & 0 \\ 0 & P_{2}^{-T} \end{bmatrix}=\begin{bmatrix} I & 0 \\ -P_{2}^{-1} & I \end{bmatrix}\begin{bmatrix} I & P_{2}-I \\ 0 & I \end{bmatrix}\begin{bmatrix} I & 0 \\ I & I \end{bmatrix}\begin{bmatrix} I & P_{2}^{-1}-I \\ 0 & I \end{bmatrix},
\end{equation*}
so that
\begin{equation*}
\begin{split}
H=&\begin{bmatrix} I & -P_{1} \\ 0 & I \end{bmatrix}\begin{bmatrix} I & 0 \\ P_{1}^{-1}-I & I \end{bmatrix}\begin{bmatrix} I & I \\ 0 & I \end{bmatrix}\begin{bmatrix} I & 0 \\ P_{1}-I-P_{2}^{-1} & I \end{bmatrix} \\
&\cdot\begin{bmatrix} I & P_{2}-I \\ 0 & I \end{bmatrix}\begin{bmatrix} I & 0 \\ I & I \end{bmatrix}\begin{bmatrix} I & P_{2}^{-1}-I \\ 0 & I \end{bmatrix}.
\end{split}
\end{equation*}
Up to now, we have obtained the expressions of $S_{i}$ exactly. For $T_{i}$, we just need to express $H^{T}$ in the above way and return to $H$ by transposition operator again.
\end{proof}
\begin{remark}
As shown in the proof of Lemma \ref{lem:decomp}, some of the blocks can be required to be $I_{d}$ for reducing the degree of freedom of factorization, such as $S_{2},S_{5}$.
\end{remark}

Now we are able to provide the proof of Theorem \ref{thm:fac_thm}.
\begin{proof}
Let $H\in SP$. According to Lemma \ref{lem:unit_ULU}, we have
\begin{equation*}
H=\begin{bmatrix} I & S \\ 0 & I \end{bmatrix}\begin{bmatrix} I & 0 \\ T & I \end{bmatrix}\begin{bmatrix} I & U \\ 0 & I \end{bmatrix}\begin{bmatrix} P & 0 \\ 0 & P^{-T} \end{bmatrix}.
\end{equation*}
where $S,T,U$ are symmetric and $P$ is nonsingular. Lemma \ref{lem:decomp} shows that
\begin{equation*}
\begin{bmatrix} P & 0 \\ 0 & P^{-T} \end{bmatrix}=\begin{bmatrix} I & S_{7} \\ 0 & I \end{bmatrix}\begin{bmatrix} I & 0 \\ S_{6} & I \end{bmatrix}\cdots\begin{bmatrix} I & 0 \\ S_{2} & I \end{bmatrix}\begin{bmatrix} I & S_{1} \\ 0 & I \end{bmatrix},
\end{equation*}
so that
\begin{equation*}
\begin{split}
H&=\begin{bmatrix} I & S \\ 0 & I \end{bmatrix}\begin{bmatrix} I & 0 \\ T & I \end{bmatrix}\begin{bmatrix} I & U \\ 0 & I \end{bmatrix}\begin{bmatrix} I & S_{7} \\ 0 & I \end{bmatrix}\begin{bmatrix} I & 0 \\ S_{6} & I \end{bmatrix}\cdots\begin{bmatrix} I & 0 \\ S_{2} & I \end{bmatrix}\begin{bmatrix} I & S_{1} \\ 0 & I \end{bmatrix} \\
&=\begin{bmatrix} I & S \\ 0 & I \end{bmatrix}\begin{bmatrix} I & 0 \\ T & I \end{bmatrix}\begin{bmatrix} I & U+S_{7} \\ 0 & I \end{bmatrix}\begin{bmatrix} I & 0 \\ S_{6} & I \end{bmatrix}\cdots\begin{bmatrix} I & 0 \\ S_{2} & I \end{bmatrix}\begin{bmatrix} I & S_{1} \\ 0 & I \end{bmatrix},
\end{split}
\end{equation*}
thus $H$ has been decomposed into 9 unit triangular symplectic matrices.
\end{proof}
\begin{remark}
This factorization starts with an upper triangular matrix in right-hand side. In fact, if we decompose the $H^{T}$ into the above 9 factors, subsequently obtain an expression of $H$ starting with a lower triangular matrix by transposition.
\end{remark}

Theorem \ref{thm:fac_thm} shows an elegant expression of symplectic matrix, and it implies several important properties.
\begin{corollary} \label{cor:det_one}
The determinant of any symplectic matrix is one.
\end{corollary}
\begin{proof}
The determinant of any unit triangular symplectic matrix is one.
\end{proof}
There are many proofs for the determinant of symplectic matrix \cite{bunger2017yet,mackey2003determinant}. Compared to the conventional proofs for the determinant of symplectic matrix, such as QR-like factorization, SVD-like factorization and transvections factorization, Theorem \ref{thm:fac_thm} is the most symplectic-determinant-revealing and elementary.

\begin{corollary} \label{cor:path_con}
The matrix symplectic group $SP$ is path connected.
\end{corollary}
\begin{proof}
Given two symplectic matrices $H,G\in SP$, we can express them as
\begin{equation*}
\begin{split}
H&=\begin{bmatrix} I & S_{9} \\ 0 & I \end{bmatrix}\cdots \begin{bmatrix} I & 0 \\ S_{2} & I \end{bmatrix}\begin{bmatrix} I & S_{1} \\ 0 & I \end{bmatrix} \\
G&=\begin{bmatrix} I & T_{9} \\ 0 & I \end{bmatrix}\cdots \begin{bmatrix} I & 0 \\ T_{2} & I \end{bmatrix}\begin{bmatrix} I & T_{1} \\ 0 & I \end{bmatrix} \\
\end{split}
\end{equation*}
where $S_{i},T_{i}$ are symmetric, due to Theorem \ref{thm:fac_thm}. Define the continuous $\gamma:[0,1]\rightarrow SP$ as
\begin{equation*}
\gamma(t)=\begin{bmatrix} I & (1-t)\cdot S_{9}+t\cdot T_{9} \\ 0 & I \end{bmatrix}\cdots \begin{bmatrix} I & 0 \\ (1-t)\cdot S_{2}+t\cdot T_{2} & I \end{bmatrix}\begin{bmatrix} I & (1-t)\cdot S_{1}+t\cdot T_{1} \\ 0 & I \end{bmatrix},
\end{equation*}
then we have $\gamma(0)=H$ and $\gamma(1)=G$, thus $SP$ is path connected.
\end{proof}
\begin{corollary} \label{cor:generators}
Denote $SG=\mathcal{L}_{1}\cup \mathcal{L}_{1}^{T}$ where $\mathcal{L}_{1}^{T}\coloneqq\{H^{T}|H\in \mathcal{L}_{1}\}$, then $SG$ is a set of generators of the group $SP$.
\end{corollary}
\begin{proof}
It is obvious by Theorem \ref{thm:fac_thm}.
\end{proof}
\begin{remark}
Although Corollary \ref{cor:det_one}, \ref{cor:path_con} are well-known and can be found in many literatures such as \cite{de2006symplectic,feng2010symplectic}, here we provide the alternative and concise proofs which are deduced directly from the unit triangular factorization. It is noteworthy that Corollary \ref{cor:generators} is indeed a new feature, while the existing theory of generators of $SP$ requires the factors of diagonal symplectic matrices (\ref{eq:dia_units}). For example, \cite[p. 49, Corollary 2.40]{de2006symplectic} states that $\{J\}\cup\{unit\ lower\ triangular\ symplectic\ matrices\}\cup\{diagonal\ symplectic\ matrices\}$ is a set of generators of $SP$. In fact, Corollary \ref{cor:generators} indicates that $\{diagonal\ symplectic\ matrices\}$ can be removed from above statement.
\end{remark}

Up to now, we have studied several properties of the general symplectic matrices. Next we focus on the structured sets of symplectic matrices.
\subsection{Symplectic matrices with nonsingular left upper block}
For symplectic matrices with other structures, which are naturally subsets of $SP$, we can exploit their special structures for reducing 9 to a smaller number. Here we turn to the desirable symplectic matrices which have nonsingular left upper blocks. Denote
\begin{equation*}
SPN=\left\{H=\begin{bmatrix} A & \star \\ \star & \star \end{bmatrix}\in SP\Bigg|A\in\R^{d\times d}\ nonsingular\right\}.
\end{equation*}
We have proved that $SP=\mathcal{L}_{9}$, furthermore, it is easy to show that $SPN\subsetneq \mathcal{L}_{8}$.
\begin{theorem}
$SPN \subsetneq \mathcal{L}_{8}$. Thus any symplectic matrix with nonsingular left upper block can be factored into no more than $8$ unit triangular symplectic matrices.
\end{theorem}
\begin{proof}
With the LDU factorization and Lemma \ref{lem:decomp}, we know
\begin{equation*}
\begin{split}
H=\begin{bmatrix} A & \star \\ \star & \star \end{bmatrix}&=\begin{bmatrix} I & 0 \\ S & I \end{bmatrix}\begin{bmatrix} A & 0 \\ 0 & A^{-T} \end{bmatrix}\begin{bmatrix} I & T \\ 0 & I \end{bmatrix} \\
&=\begin{bmatrix} I & 0 \\ S & I \end{bmatrix}\begin{bmatrix} I & S_{7} \\ 0 & I \end{bmatrix}\cdots \begin{bmatrix} I & 0 \\ S_{2} & I \end{bmatrix}\begin{bmatrix} I & S_{1}+T \\ 0 & I \end{bmatrix},
\end{split}
\end{equation*}
so that $SPN\subset \mathcal{L}_{8}$. Moreover, consider $J=\begin{bmatrix} I & I \\ 0 & I \end{bmatrix}\begin{bmatrix} I & 0 \\ -I & I \end{bmatrix}\begin{bmatrix} I & I \\ 0 & I \end{bmatrix}$, we have $J\in \mathcal{L}_{3}\subset \mathcal{L}_{8}$ and $J\notin SPN$, hence $SPN \subsetneq \mathcal{L}_{8}$.
\end{proof}
Until now, we have $SPN\subsetneq \mathcal{L}_{8}\subset \mathcal{L}_{9}=SP$, however, it is not sure whether $\mathcal{L}_{8}=\mathcal{L}_{9}$ or $\mathcal{L}_{8}\subsetneq \mathcal{L}_{9}$. \cite{dopico2009parametrization} provides a proof that $SPN$ is a dense subset in the group of symplectic matrices, although it can be accomplished through general properties of algebraic manifolds. Here we present a simple proof again according to Theorem \ref{thm:nonsin_thm}, and this result immediately shows that $\mathcal{L}_{8}$ is at least dense in $\mathcal{L}_{9}=SP$.
\begin{corollary}
$SPN$ is dense in $SP$, i.e., $\overline{SPN}=SP$.
\end{corollary}
\begin{proof}
Theorem \ref{thm:nonsin_thm} shows that, for any $H\in SP$, there exists a symmetric $S$ such that
\begin{equation*}
    H=\begin{bmatrix} I & \lambda S \\ 0 & I \end{bmatrix}\begin{bmatrix} P_{\lambda} & \star \\ \star & \star \end{bmatrix}
\end{equation*}
holds with a nonsingular $P_{\lambda}$ for all $\lambda\neq0$. Therefore
\begin{equation*}
    H=\lim_{\lambda\rightarrow 0}\begin{bmatrix} I & -\lambda S \\ 0 & I \end{bmatrix}H=\lim_{\lambda\rightarrow 0}\begin{bmatrix} P_{\lambda} & \star \\ \star & \star \end{bmatrix}\in \overline{SPN},
\end{equation*}
which means $SP\subset \overline{SPN}$, consequently $SP=\overline{SPN}$.
\end{proof}
\begin{corollary}
$\mathcal{L}_{8}$ is dense in $\mathcal{L}_{9}$, i.e., $\overline{\mathcal{L}_{8}}=\mathcal{L}_{9}$.
\end{corollary}
\begin{proof}
$SPN$ is dense in $SP$.
\end{proof}

\subsection{Positive definite symplectic matrices}
Denote the collection of symmetric positive definite symplectic matrices by
\begin{equation*}
SPP=\{H\in SP| H\ symmetric\ positive\ definite \}.
\end{equation*}
\begin{theorem}\label{thm:spp_thm}
$SPP \subsetneq \mathcal{L}_{4}$. Thus any symmetric positive definite symplectic matrix can be factored into no more than $4$ unit triangular symplectic matrices.
\end{theorem}

\begin{proof}
If $H$ is symmetric positive definite, then all  principal submatrices of $H$ are symmetric positive definite. Denote that $H=\begin{bmatrix} H_{1} & H_{2} \\ H_{3} & H_{4} \end{bmatrix}\in SPP$, here $H_1$ is positive definite, thus nonsingular. So that in the process of decomposition, instead of going through Theorem \ref{thm:nonsin_thm}, we just apply the LDU factorization to H, and  get
\begin{equation*}
    H=\begin{bmatrix} I & 0 \\ H_3H_1^{-1} & I \end{bmatrix}\begin{bmatrix} H_1 & 0 \\ 0 & H_1^{-T} \end{bmatrix}\begin{bmatrix} I & H_1^{-1}H_2 \\ 0 & I \end{bmatrix},
\end{equation*}
where $H_3H_1^{-1}$ and $H_1^{-1}H_2$ are symmetric matrices. Moreover, because of the symmetry of $H$, $H_3H_1^{-1}=(H_1^{-1}H_2)^T=H_1^{-1}H_2$. Therefore, every positive definite symplectic matrix can be written as
\begin{equation} \label{eq:SPP_LDU}
    H=\begin{bmatrix} I & 0 \\ S & I \end{bmatrix}\begin{bmatrix} P & 0 \\ 0 & P^{-T} \end{bmatrix}\begin{bmatrix} I & S \\ 0 & I \end{bmatrix}
\end{equation}
with $P$ symmetric positive definite.
Since $P$ is already symmetric here, we do not have to decompose it into the product of two symmetric matrices, and just have the factorization
\begin{equation*}
\begin{bmatrix} P & 0 \\ 0 & P^{-T} \end{bmatrix}=\begin{bmatrix} I & 0 \\ -P^{-1} & I \end{bmatrix}\begin{bmatrix} I & P-I \\ 0 & I \end{bmatrix}\begin{bmatrix} I & 0 \\ I & I \end{bmatrix}\begin{bmatrix} I & P^{-1}-I \\ 0 & I \end{bmatrix},
\end{equation*}
as a consequence
\begin{equation*}
\begin{split}
    H&=\begin{bmatrix} I & 0 \\ S & I \end{bmatrix}\begin{bmatrix} I & 0 \\ -P^{-1} & I \end{bmatrix}\begin{bmatrix} I & P-I \\ 0 & I \end{bmatrix}\begin{bmatrix} I & 0 \\ I & I \end{bmatrix}\begin{bmatrix} I & P^{-1}-I \\ 0 & I \end{bmatrix}\begin{bmatrix} I & S \\ 0 & I \end{bmatrix}\\
    &=\begin{bmatrix} I & 0 \\ S-P^{-1} & I \end{bmatrix}\begin{bmatrix} I & P-I \\ 0 & I \end{bmatrix}\begin{bmatrix} I & 0 \\ I & I \end{bmatrix}\begin{bmatrix} I & S+P^{-1}-I \\ 0 & I \end{bmatrix},
\end{split}
\end{equation*}
thus $H\in \mathcal{L}_{4}$.

On the other hand, $J=\begin{bmatrix} I & I \\ 0 & I \end{bmatrix}\begin{bmatrix} I & 0 \\ -I & I \end{bmatrix}\begin{bmatrix} I & I \\ 0 & I \end{bmatrix}\in \mathcal{L}_{3}\subset \mathcal{L}_{4}$ while $J\notin SPP$, hence $SPP\subsetneq \mathcal{L}_{4}$.
\end{proof}
\begin{remark}
The above factorization starts with an upper triangular matrix in right-hand side. Similar to the proof of Theorem \ref{thm:spp_thm} , we can use the LDU factorization in the case of $H_4$ invertible,
\begin{equation}\label{eq:LDU2}
    H=\begin{bmatrix} I & S \\ 0 & I \end{bmatrix}\begin{bmatrix} P & 0 \\ 0 & P^{-T} \end{bmatrix}\begin{bmatrix} I & 0 \\ S & I \end{bmatrix},
\end{equation}
 to obtain another decomposition where the first one is a lower triangular matrix.
\end{remark}

Going a step further, we can prove that $\mathcal{L}_{4}$ is optimal, i.e., $SPP\nsubseteq \mathcal{L}_{3}$.
\begin{theorem} \label{thm:spp_l3}
$SPP\nsubseteq \mathcal{L}_{3}$.
\end{theorem}
\begin{proof}
Let $G=\begin{bmatrix} 2I & 0 \\ 0 & \frac{1}{2}I \end{bmatrix}\in SPP$. If there exist $S_{1},S_{2},S_{3}$ such that
\begin{equation*}
G=\begin{bmatrix} I & S_{3} \\ 0 & I \end{bmatrix}\begin{bmatrix} I & 0 \\ S_{2} & I \end{bmatrix}\begin{bmatrix} I & S_{1} \\ 0 & I \end{bmatrix}=\begin{bmatrix} I+S_{3}S_{2} & S_{1}+S_{3}S_{2}S_{1}+S_{3} \\ S_{2} & I+S_{2}S_{1} \end{bmatrix},
\end{equation*}
then we have $S_{2}=0$, $I=2I$, which is a contradiction.
\end{proof}

\subsection{Symplectic M-matrices}
M-matrices occur very often in a wide variety of areas including finite difference methods for partial differential equations, economics, probability and statistics \cite{berman1994nonnegative}. There are many equivalent definitions of M-matrices, we adopt the following one.
\begin{definition} \label{def:M-matrices}
$A\in \R^{n\times n}$ is an M-matrix if $a_{ij}\leq0$ for $i \neq j$ and $Re(\lambda)>0$ for every eigenvalue $\lambda$ of $A$.
\end{definition}
The next result is a factorization of the symplectic M-matrices. It appears in \cite{dopico2009parametrization} and is the key result on how symplectic M-matrices are decomposed into the product of several unit triangular symplectic matrices.
\begin{theorem} \label{thm:fac_M}
The set of $2d$-by-$2d$ symplectic M-matrices is
\begin{equation*}
SPM=\left\{\begin{bmatrix} I & 0 \\ H & I \end{bmatrix}\begin{bmatrix} D & 0 \\ 0 & D^{-1} \end{bmatrix}\begin{bmatrix} I & K \\ 0 & I \end{bmatrix}\Bigg|\begin{aligned}&D \in \R^{d\times d}\ positive \ diagonal \\ &H=H^{T}\leq 0 \\ &K=K^{T}\leq 0 \\ &HDK \ diagonal \end{aligned}\right\},
\end{equation*}
where the inequalities $H \leq 0$ and $K \leq 0$ mean that $h_{ij} \leq 0$ and $k_{ij} \leq 0$ for all $i,j$.
\end{theorem}
Here we treat $D$ as an ordinary symmetric matrix, as in Theorem \ref{thm:spp_thm}, any symplectic M-matrix can be decomposed as the product of four triangular blocks.
\begin{theorem}
$SPM \subsetneq \mathcal{L}_{4}$. Thus any symplectic M-matrix can be factored into no more than $4$ unit triangular symplectic matrices.
\end{theorem}
\begin{proof}
It is easy to show $SPM\subset \mathcal{L}_{4}$ based on Theorem \ref{thm:fac_M}. Moreover, $-I_{2d}\in \mathcal{L}_{4}$ while $-I_{2d}\notin SPM$, hence $SPM\subsetneq \mathcal{L}_{4}$.
\end{proof}
\begin{theorem}
$SPM\nsubseteq \mathcal{L}_{3}$.
\end{theorem}
\begin{proof}
It is similar to the proof of Theorem \ref{thm:spp_l3}.
\end{proof}

\section{Unconstrained parametrization and optimization} \label{sec:parametrization}
In this section, we first consider the unconstrained parametrization of the symplectic matrices or its subsets. For a set $\Omega$, we expect to find a map $\phi:\R^{n}\rightarrow \Omega$ which is smooth and surjective. Therefore one may be able to deal with some problems on unconstrained parameter space $\R^n$ instead of $\Omega$ by employing $\phi$. In the last subsection, we turn to the topic about how the unconstrained parametrization impacts on the optimization problems with symplectic constraints. In order to represent the free parameters of a symmetric matrix, we define the map $Pa$ for extracting the lower triangular parameters as $Pa(S)=(s_{11},s_{21},s_{22},s_{31},\cdots,s_{dd})$, where $S=(s_{ij})\in \R^{d\times d}$, $S^{T}=S$.

\subsection{Parametrization of general symplecic matrices}
Based on Theorem \ref{thm:fac_thm}, one can parametrize $SP$ efficiently. Take symmetric $S_{1},S_{2},\cdots,S_{9}\in \R^{d\times d}$, then
\begin{equation} \label{eq:SP_para}
H(Pa(S_{1}),\cdots,Pa(S_{9}))=\begin{bmatrix} I & S_{9} \\ 0 & I \end{bmatrix}\cdots \begin{bmatrix} I & 0 \\ S_{2} & I \end{bmatrix}\begin{bmatrix} I & S_{1} \\ 0 & I \end{bmatrix} \\
\end{equation}
can represent any symplectic matrix when $Pa(S_{1}),\cdots,Pa(S_{9})$ vary. Note that here $Pa(S_{i})$ are free parameters without any constraint, thus $H$ is a surjective map from $\R^{\frac{9d(d+1)}{2}}$ into $SP$. With (\ref{eq:SP_para}), it is easy to write out the inverse element and the transpose of $H$, i.e.,
\begin{equation} \label{eq:inv}
H(Pa(S_{1}),\cdots,Pa(S_{9}))^{-1}=\begin{bmatrix} I & -S_{1} \\ 0 & I \end{bmatrix}\cdots \begin{bmatrix} I & 0 \\ -S_{8} & I \end{bmatrix}\begin{bmatrix} I & -S_{9} \\ 0 & I \end{bmatrix},
\end{equation}
\begin{equation} \label{eq:trans}
H(Pa(S_{1}),\cdots,Pa(S_{9}))^{T}=\begin{bmatrix} I & 0 \\ S_{1} & I \end{bmatrix}\cdots \begin{bmatrix} I & S_{8} \\ 0 & I \end{bmatrix}\begin{bmatrix} I & 0 \\ S_{9} & I \end{bmatrix}.
\end{equation}
The concise expressions of the inverse element and the transpose make it possible to construct more complex structured symplectic matrices.
\subsection{Parametrization of positive definite symplectic matrices}
By Theorem \ref{thm:spp_thm}, we know that $SPP$ is a subset of $\mathcal{L}_4$. Since $\mathcal{L}_n$ cannot guarantee the symmetry for its elements, a new collection of symplectic matrices is needed. The following theorem will show how to construct the unconstrained parametrization of $SPP$. Denote
\begin{equation*}
\mathcal{L}_{n} ^2= \{ LL^T| L \in \mathcal{L}_{n}\}.
\end{equation*}
By (\ref{eq:inv}) and (\ref{eq:trans}), we know $L$ is invertible and $L^{T}$ is still symplectic, so that $LL^T$ is a positive definite  symplectic matrix. Thus $\mathcal{L}_{m}^2 \subset \mathcal{L}_{n}^2\subset SPP$ for all integers $1\leq m\leq n$.
\begin{theorem} \label{thm:sppfac_thm}
$SPP=\mathcal{L}_{4}^2$. Thus any positive definite symplectic matrix can be represented as the product of a matrix which can be factored into no more than $4$ unit triangular symplectic matrices and its transpose.
\end{theorem}
\begin{proof}
According to (\ref{eq:SPP_LDU}) in the proof of Theorem \ref{thm:spp_thm}, every positive definite symplectic matrix can be written as
\begin{equation*}
    H=\begin{bmatrix} I & 0 \\ S & I \end{bmatrix}\begin{bmatrix} P & 0 \\ 0 & P^{-T} \end{bmatrix}\begin{bmatrix} I & S \\ 0 & I \end{bmatrix}
\end{equation*}
with $P$ symmetric positive definite. Let $P^{1/2}$ be the unique positive definite square root of $P$. $P=P^{1/2}P^{1/2}$ implies that $P^{-1}=(P^{1/2})^{-1}(P^{1/2})^{-1}$, then
\begin{equation*}
\begin{split}
    H&=\begin{bmatrix} I & 0 \\ S & I \end{bmatrix}\begin{bmatrix} P^{1/2} & 0 \\ 0 & (P^{1/2})^{-1}\end{bmatrix}\begin{bmatrix} P^{1/2} & 0 \\ 0 & (P^{1/2})^{-1} \end{bmatrix}\begin{bmatrix} I & S \\ 0 & I \end{bmatrix}\\
    &= \left(\begin{bmatrix} I & 0 \\ S & I \end{bmatrix}\begin{bmatrix} P^{1/2} & 0 \\ 0 & (P^{1/2})^{-1} \end{bmatrix}\right) \left(\begin{bmatrix} I & 0 \\ S & I \end{bmatrix}\begin{bmatrix} P^{1/2} & 0 \\ 0 & (P^{1/2})^{-1} \end{bmatrix} \right)^T.
\end{split}
\end{equation*}
Just like what we did in Theorem \ref{thm:spp_thm}, taking into account Lemma \ref{lem:decomp} and the symmetry of $P^{1/2}$, we have
\begin{equation*}
\begin{split}
&\begin{bmatrix} I & 0 \\ S & I \end{bmatrix}\begin{bmatrix} P^{1/2} & 0 \\ 0 & (P^{1/2})^{-1} \end{bmatrix} \\
=&\begin{bmatrix} I & 0 \\ S & I \end{bmatrix}\begin{bmatrix} I & 0 \\ -(P^{1/2})^{-1} & I \end{bmatrix}\begin{bmatrix} I & P^{1/2}-I \\ 0 & I \end{bmatrix}\begin{bmatrix} I & 0 \\ I & I \end{bmatrix}\begin{bmatrix} I & (P^{1/2})^{-1}-I \\ 0 & I \end{bmatrix} \\
=&\begin{bmatrix} I & 0 \\ S-(P^{1/2})^{-1} & I \end{bmatrix}\begin{bmatrix} I & P^{1/2}-I \\ 0 & I \end{bmatrix}\begin{bmatrix} I & 0 \\ I & I \end{bmatrix}\begin{bmatrix} I & (P^{1/2})^{-1}-I \\ 0 & I \end{bmatrix}, \\
\end{split}
\end{equation*}
hence we have obtained the expression of $L$ exactly.
\end{proof}
Based on Theorem \ref{thm:sppfac_thm}, one can represent $SPP$ by parameterizing $\mathcal{L}_4$. Take symmetric $S_1,S_2,S_3,S_4 \in \R^{d \times d}$, then denote
\begin{equation*}
\begin{split}
&H(Pa(S_{1}),\cdots,Pa(S_{4}))=LL^{T},\\
&L=\begin{bmatrix} I & 0 \\ S_{4} & I \end{bmatrix}\begin{bmatrix} I & S_{3} \\ 0 & I \end{bmatrix}\begin{bmatrix} I & 0 \\ S_{2} & I \end{bmatrix}\begin{bmatrix} I & S_{1} \\ 0 & I \end{bmatrix}. \\
\end{split}
\end{equation*}
Here $Pa(S_{i})$ are free parameters without any constraint, thus $H$ is a surjective map from $\R^{2d(d+1)}$ into $SPP$.

\subsection{Parametrization of singular symplectic matrices}
There have been some studies on the structure of singular symplectic matrices \cite{long1990maslov,long1991structure}.
\begin{definition} \label{def:singular_SP}
A symplectic matrix $H\in SP$ is singular if $det(H-I)=0$.
\end{definition}
Denote all the singular symplectic matrices by
\begin{equation*}
SPS=\{H\in SP|H\ singular \}.
\end{equation*}
A symplectic matrix $H$ is singular when $1$ is one of its eigenvalues, and there is a $0\neq u \in \R^{2d}$ such that $Hu=u$ where $u$ is called the fixed vector of $H$. Theorem \ref{thm:fac_thm} prompts us to consider the following question: If there are several linearly independent vectors $u_{1},u_{2},...,u_{n}\in\R^{2d},n\leq 2d$, how to search for a symplectic matrix $H$ such that $H u_i=u_i$?

Let us consider a slightly simpler case. Denote
\begin{equation*}
U=[u_{1},u_{2},\cdots,u_{n}]=\begin{bmatrix} V_{1} \\ V_{2} \end{bmatrix}\in\R^{2d\times n},\quad V_{1},V_{2}\in\R^{d\times n},
\end{equation*}
where $U$ is a full-rank upper triangular matrix, and assume that $rank(V_{1})=r_{1}$, $rank(V_{2})=r_{2}$. Now let
\begin{equation*}
\begin{split}
&\mathcal{V}_{1}=\left\{\begin{bmatrix} I & 0 \\ S & I\end{bmatrix}\Bigg|S=\begin{bmatrix} O_{r_{1}} & 0 \\ 0 & T \end{bmatrix},\quad T\in\R^{(d-r_{1})\times(d-r_{1})}\ symmetric\right\}, \\
&\mathcal{V}_{2}=\left\{\begin{bmatrix} I & S \\ 0 & I\end{bmatrix}\Bigg|S=\begin{bmatrix} O_{r_{2}} & 0 \\ 0 & T \end{bmatrix},\quad T\in\R^{(d-r_{2})\times(d-r_{2})}\ symmetric\right\}, \\
&\mathcal{V}_{U}=\{H_{k}\cdots H_{2}H_{1}|H_{i}\in\mathcal{V}_{1}\cup\mathcal{V}_{2},k\geq1\}.
\end{split}
\end{equation*}
Note that
\begin{equation*}
\left\{
\begin{aligned}
&r_{1}=n,r_{2}=0\quad &when\ 1\leq n<d \\
&r_{1}=d,r_{2}=n-d\quad &when\ d\leq n\leq 2d \\
\end{aligned}
\right..
\end{equation*}
Given any $H\in \mathcal{V}_{U}$, it is easy to verify that $HU=U$. For instance, if $H\in \mathcal{V}_{1}$, then
\begin{equation*}
HU=\begin{bmatrix} I & 0 \\ \begin{bmatrix} O_{r_{1}} & 0 \\ 0 & T \end{bmatrix} & I\end{bmatrix}\begin{bmatrix} V_{1} \\ V_{2} \end{bmatrix}=\begin{bmatrix} V_{1} \\ \begin{bmatrix} O_{r_{1}} & 0 \\ 0 & T \end{bmatrix}V_{1}+V_{2} \end{bmatrix}=\begin{bmatrix} V_{1} \\ V_{2} \end{bmatrix}=U. \\
\end{equation*}
Note that the last $d-r_{1}$ rows of $V_{1}$ are all $0$.

Denote the set of all the symplectic matrices $H$ satisfying $HU=U$ by $\mathcal{W}_{U}$.
It is easy to see that $\mathcal{W}_{U}$ and $\mathcal{V}_{U}$ are both groups, moreover, $\mathcal{V}_{U}$ is a subgroup of $\mathcal{W}_{U}$. Whether $\mathcal{V}_{U}=\mathcal{W}_{U}$ or $\mathcal{V}_{U}\subsetneq\mathcal{W}_{U}$ is an interesting problem to be concerned with. At least, in the case of $n=0$, Theorem \ref{thm:fac_thm} immediately points out that $\mathcal{V}_{U}=\mathcal{W}_{U}$ since both of $\mathcal{V}_{U}$ and $\mathcal{W}_{U}$ equal to $SP$, while in the case of $n=2d$, $\mathcal{V}_{U}$ and $\mathcal{W}_{U}$ degenerate both into $\{I_{2d}\}$. A theorem is proposed as follows.
\begin{theorem} \label{thm:SPS_thm}
$\mathcal{V}_{U}=\mathcal{W}_{U}$. Here $U\in\R^{2d\times n}$ ($1\leq n \leq 2d$) is any full-rank upper triangular matrix. Furthermore,
\begin{equation*}
\left\{
\begin{aligned}
&\mathcal{V}_{U}=\{H_{10}\cdots H_{2}H_{1}|H_{2i-1}\in\mathcal{V}_{2},H_{2i}\in\mathcal{V}_{1}\}\quad &when\ 1\leq n<d \\
&\mathcal{V}_{U}=\{H|H\in\mathcal{V}_{2}\}\quad &when\ d\leq n\leq 2d \\
\end{aligned}
\right..
\end{equation*}
\end{theorem}
\begin{proof}
\textbf{1.} Case $d\leq n\leq 2d$. Denote
\begin{equation*}
U=\begin{bmatrix}U_{11} & U_{12} \\ 0 & U_{22}\end{bmatrix}=\begin{bmatrix}U_{11} & U_{12} \\ 0 & \begin{bmatrix}\widetilde{U}_{22} \\ 0\end{bmatrix}\end{bmatrix}
\end{equation*}
where $U_{11}\in\R^{d\times d},U_{12}\in\R^{d\times (n-d)},U_{22}\in\R^{d\times (n-d)},\widetilde{U}_{22}\in\R^{(n-d)\times(n-d)}$. Since $U$ is full-rank upper triangular, $U_{11}$ and $\widetilde{U}_{22}$ are nonsingular upper triangular matrices. For any $H=\begin{bmatrix}A_{1} & B_{1} \\ A_{2} & B_{2} \end{bmatrix}\in \mathcal{W}_{U}$,
\begin{equation*}
\begin{split}
&HU=U\Rightarrow\begin{bmatrix}A_{1} & B_{1} \\ A_{2} & B_{2} \end{bmatrix}\begin{bmatrix}U_{11} & U_{12} \\ 0 & U_{22}\end{bmatrix}=\begin{bmatrix}U_{11} & U_{12} \\ 0 & U_{22}\end{bmatrix} \\
&\Rightarrow\begin{bmatrix}A_{1}U_{11} & A_{1}U_{12}+B_{1}U_{22} \\ A_{2}U_{11} & A_{2}U_{12}+B_{2}U_{22} \end{bmatrix}=\begin{bmatrix}U_{11} & U_{12} \\ 0 & U_{22}\end{bmatrix}\Rightarrow A_{1}=I,\ A_{2}=0.
\end{split}
\end{equation*}
With $A_{1}=I,A_{2}=0$ and the third item of Property \ref{pro:smat_pro}, we find that $H=\begin{bmatrix}I & B_{1} \\ 0 & I \end{bmatrix}$. Moreover,
\begin{equation*}
\begin{split}
&A_{1}U_{12}+B_{1}U_{22}=U_{12}\Rightarrow B_{1}U_{22}=0\Rightarrow\begin{bmatrix}B_{11} & B_{12} \\ B_{21} & B_{22}\end{bmatrix}\begin{bmatrix}\widetilde{U}_{22} \\ 0\end{bmatrix}=0 \\
&\Rightarrow\begin{bmatrix}B_{11}\widetilde{U}_{22} \\ B_{21}\widetilde{U}_{22}\end{bmatrix}=0\Rightarrow B_{11}=0,\ B_{21}=0,
\end{split}
\end{equation*}
and the symmetry of $B_{1}$ shows that $B_{12}=B_{21}=0$. Hence
\begin{equation*}
H=\begin{bmatrix}I & \begin{bmatrix}O_{n-d} & 0 \\ 0 & B_{22} \end{bmatrix} \\ 0 & I \end{bmatrix}\in\mathcal{V}_{2}\subset\mathcal{V}_{U},
\end{equation*}
which means $\mathcal{W}_{U}\subset \mathcal{V}_{U}$, consequently $\mathcal{V}_{U}=\mathcal{W}_{U}$.

\textbf{2.} Case $1\leq n<d$. Denote
\begin{equation*}
U=\begin{bmatrix}U_{1}\\0\end{bmatrix}=\begin{bmatrix}\begin{bmatrix}\widetilde{U}_{1}\\0\end{bmatrix}\\0\end{bmatrix},\quad U_{1}\in\R^{d\times n},\widetilde{U}_{1}\in\R^{n\times n},
\end{equation*}
where $\widetilde{U}_{1}$ is nonsingular upper triangular. For any $H=\begin{bmatrix}A & \star \\ B & \star \end{bmatrix}\in \mathcal{W}_{U}$,
\begin{equation*}
\begin{split}
&HU=U\Rightarrow\begin{bmatrix}A & \star \\ B & \star \end{bmatrix}\begin{bmatrix}U_{1}\\0\end{bmatrix}=\begin{bmatrix}U_{1}\\0\end{bmatrix}\Rightarrow
\begin{bmatrix}\begin{bmatrix}A_{11}&A_{12}\\A_{21}&A_{22}\end{bmatrix} & \star \\ \begin{bmatrix}B_{11}&B_{12}\\B_{21}&B_{22}\end{bmatrix} & \star \end{bmatrix}\begin{bmatrix}\begin{bmatrix}\widetilde{U}_{1}\\0\end{bmatrix}\\0\end{bmatrix}=
\begin{bmatrix}\begin{bmatrix}\widetilde{U}_{1}\\0\end{bmatrix}\\0\end{bmatrix} \\
&\Rightarrow\begin{bmatrix}\begin{bmatrix}A_{11}&A_{12}\\A_{21}&A_{22}\end{bmatrix}\cdot\begin{bmatrix}\widetilde{U}_{1}\\0\end{bmatrix} \\ \begin{bmatrix}B_{11}&B_{12}\\B_{21}&B_{22}\end{bmatrix}\cdot\begin{bmatrix}\widetilde{U}_{1}\\0\end{bmatrix}  \end{bmatrix}=\begin{bmatrix}\begin{bmatrix}\widetilde{U}_{1}\\0\end{bmatrix}\\0\end{bmatrix}\Rightarrow
\begin{bmatrix}\begin{bmatrix}A_{11}\widetilde{U}_{1}\\A_{21}\widetilde{U}_{1}\end{bmatrix} \\ \begin{bmatrix}B_{11}\widetilde{U}_{1}\\B_{21}\widetilde{U}_{1}\end{bmatrix}  \end{bmatrix}=\begin{bmatrix}\begin{bmatrix}\widetilde{U}_{1}\\0\end{bmatrix}\\0\end{bmatrix} \\
&\Rightarrow A_{11}=I_{n},\ A_{21}=0,\ B_{11}=O_{n},\ B_{21}=0.
\end{split}
\end{equation*}
Property \ref{pro:smat_pro} leads to
\begin{equation*}
\begin{split}
&\begin{bmatrix}A_{11}&A_{12}\\A_{21}&A_{22}\end{bmatrix}^{T}\begin{bmatrix}B_{11}&B_{12}\\B_{21}&B_{22}\end{bmatrix}=
\begin{bmatrix}B_{11}&B_{12}\\B_{21}&B_{22}\end{bmatrix}^{T}\begin{bmatrix}A_{11}&A_{12}\\A_{21}&A_{22}\end{bmatrix} \\
&\Rightarrow \begin{bmatrix}I_{n}&A_{12}\\0&A_{22}\end{bmatrix}^{T}\begin{bmatrix}O_{n}&B_{12}\\0&B_{22}\end{bmatrix}=
\begin{bmatrix}O_{n}&B_{12}\\0&B_{22}\end{bmatrix}^{T}\begin{bmatrix}I_{n}&A_{12}\\0&A_{22}\end{bmatrix} \\
&\Rightarrow \begin{bmatrix}O_{n}&B_{12}\\0&A_{12}^{T}B_{12}+A_{22}^{T}B_{22}\end{bmatrix}=
\begin{bmatrix}O_{n}&0\\B_{12}^{T}&B_{12}^{T}A_{12}+B_{22}^{T}A_{22}\end{bmatrix} \\
&\Rightarrow B_{12}=0,\ A_{22}^{T}B_{22}=B_{22}^{T}A_{22}.
\end{split}
\end{equation*}
Now we have
\begin{equation*}
H=\begin{bmatrix}\begin{bmatrix}A_{11}&A_{12}\\A_{21}&A_{22}\end{bmatrix} & \star \\ \begin{bmatrix}B_{11}&B_{12}\\B_{21}&B_{22}\end{bmatrix} & \star \end{bmatrix}=\begin{bmatrix}\begin{bmatrix}I_{n}&A_{12}\\0&A_{22}\end{bmatrix} & \star \\ \begin{bmatrix}O_{n}&0\\0&B_{22}\end{bmatrix} & \star \end{bmatrix},\quad A_{22}^{T}B_{22}=B_{22}^{T}A_{22}.
\end{equation*}
Since $rank(\begin{bmatrix}A\\B\end{bmatrix})=d$, we know $rank(\begin{bmatrix}A_{22}\\B_{22}\end{bmatrix})=d-rank(\begin{bmatrix}I_{n} & A_{12}\end{bmatrix})=d-n$, thus $G=\begin{bmatrix}A_{22}\\B_{22}\end{bmatrix}$ is full-rank. Property \ref{pro:part_SP} tells us that $G$ is a part of a $2(d-n)$-by-$2(d-n)$ symplectic matrix $\widetilde{G}=\begin{bmatrix}A_{22}&\star\\B_{22}&\star\end{bmatrix}$. Similar to the proof of Theorem \ref{thm:nonsin_thm}, there exist a symmetric $S$ and nonsingular $P,Q$ such that
\begin{equation*}
\begin{bmatrix}P&0\\0&P^{-T}\end{bmatrix}\begin{bmatrix}I_{d-n}&0\\S&I_{d-n}\end{bmatrix}\widetilde{G}\begin{bmatrix}Q&0\\0&Q^{-T}\end{bmatrix}=
\begin{bmatrix} G_{11}&\star\\I_{d-n}&\star\end{bmatrix}
\end{equation*}
where $G_{11}$ is symmetric. We have
\begin{equation*}
\begin{split}
&\begin{bmatrix}I&0\\\begin{bmatrix}O_{n}&0\\0&-I_{d-n}\end{bmatrix}&I\end{bmatrix}
\begin{bmatrix}I&\begin{bmatrix}O_{n}&-A_{12}Q\\-Q^{T}A_{12}^{T}&I_{d-n}-G_{11}\end{bmatrix}\\0&I\end{bmatrix}
\begin{bmatrix}\begin{bmatrix}I_{n}&0\\0&P\end{bmatrix}&0\\0&\begin{bmatrix}I_{n}&0\\0&P^{-T}\end{bmatrix}\end{bmatrix} \\
&\cdot\begin{bmatrix}I&0\\\begin{bmatrix}O_{n}&0\\0&S\end{bmatrix}&I\end{bmatrix}\begin{bmatrix}\begin{bmatrix}I_{n}&A_{12}\\0&A_{22}\end{bmatrix} & \star \\ \begin{bmatrix}O_{n}&0\\0&B_{22}\end{bmatrix} & \star \end{bmatrix}\begin{bmatrix}\begin{bmatrix}I_{n}&0\\0&Q\end{bmatrix}&0\\0&\begin{bmatrix}I_{n}&0\\0&Q^{-T}\end{bmatrix}\end{bmatrix} \\
=&\begin{bmatrix}I&0\\\begin{bmatrix}O_{n}&0\\0&-I_{d-n}\end{bmatrix}&I\end{bmatrix}
\begin{bmatrix}I&\begin{bmatrix}O_{n}&-A_{12}Q\\-Q^{T}A_{12}^{T}&I_{d-n}-G_{11}\end{bmatrix}\\0&I\end{bmatrix}
\begin{bmatrix}\begin{bmatrix}I_{n}&A_{12}Q\\0&G_{11}\end{bmatrix} & \star \\ \begin{bmatrix}O_{n}&0\\0&I_{d-n}\end{bmatrix} & \star \end{bmatrix} \\
=&\begin{bmatrix}I&0\\\begin{bmatrix}O_{n}&0\\0&-I_{d-n}\end{bmatrix}&I\end{bmatrix}
\begin{bmatrix}\begin{bmatrix}I_{n}&0\\0&I_{d-n}\end{bmatrix} & \star \\ \begin{bmatrix}O_{n}&0\\0&I_{d-n}\end{bmatrix} & \star \end{bmatrix}
=\begin{bmatrix}I & \star\\ 0 & \star\end{bmatrix}=\begin{bmatrix}I & T\\ 0 & I\end{bmatrix}
\end{split}
\end{equation*}
where $T$ is symmetric, hence
\begin{equation*}
\begin{split}
H=&\begin{bmatrix}I&0\\\begin{bmatrix}O_{n}&0\\0&-S\end{bmatrix}&I\end{bmatrix}
\begin{bmatrix}\begin{bmatrix}I_{n}&0\\0&P^{-1}\end{bmatrix}&0\\0&\begin{bmatrix}I_{n}&0\\0&P^{T}\end{bmatrix}\end{bmatrix}
\begin{bmatrix}I&\begin{bmatrix}O_{n}&A_{12}Q\\Q^{T}A_{12}^{T}&G_{11}-I_{d-n}\end{bmatrix}\\0&I\end{bmatrix} \\
&\cdot\begin{bmatrix}I&0\\\begin{bmatrix}O_{n}&0\\0&I_{d-n}\end{bmatrix}&I\end{bmatrix}\begin{bmatrix}I & T\\ 0 & I\end{bmatrix}
\begin{bmatrix}\begin{bmatrix}I_{n}&0\\0&Q^{-1}\end{bmatrix}&0\\0&\begin{bmatrix}I_{n}&0\\0&Q^{T}\end{bmatrix}\end{bmatrix}.
\end{split}
\end{equation*}
Taking into account Remark \ref{rem:moving} and Lemma \ref{lem:decomp}, we can easily factor $H$ into $10$ unit triangular symplectic matrices from $\mathcal{V}_{1}$ and $\mathcal{V}_{2}$. Therefore $\mathcal{W}_{U}\subset\mathcal{V}_{U}$, consequently $\mathcal{W}_{U}=\mathcal{V}_{U}$.
\end{proof}

Now return to the topic about how to parameterize singular symplectic matrices. Theorem \ref{thm:SPS_thm} gives a specific structure of singular symplectic matrices with many linearly independent vectors as their eigenvectors, especially when $U=e_{1}=[1,0,\cdots,0]^{T}$, we obtain $\mathcal{W}_{e_{1}}=\mathcal{V}_{e_{1}}$ whose structure is clear to be freely parameterized.
\begin{corollary}
The set of $2d$-by-$2d$ singular symplectic matrices is
\begin{equation*}
\begin{split}
SPS=&\Bigg\{Q\begin{bmatrix}I&0\\\begin{bmatrix}0&0\\0&S_{10}\end{bmatrix}&I\end{bmatrix}\begin{bmatrix}I&S_{9}\\0&I\end{bmatrix}\cdots
\begin{bmatrix}I&0\\\begin{bmatrix}0&0\\0&S_{2}\end{bmatrix}&I\end{bmatrix}\begin{bmatrix}I&S_{1}\\0&I\end{bmatrix}Q^{-1} \\
&\Bigg|S_{2i-1}\in\R^{d\times d}\ symmetric,\ S_{2i}\in\R^{(d-1)\times(d-1)}\ symmetric,\ Q\ symplectic\Bigg\}.
\end{split}
\end{equation*}
One can express $Q$ as the product of $9$ unit triangular symplectic matrices if needed.
\end{corollary}
\begin{proof}
For any $H=Q\begin{bmatrix}I&0\\\begin{bmatrix}0&0\\0&S_{10}\end{bmatrix}&I\end{bmatrix}\cdots\begin{bmatrix}I&S_{1}\\0&I\end{bmatrix}Q^{-1}\in SPS$, let $u=Qe_{1}\neq 0$. Then it is easy to check that $Hu=u$, hence $H$ is singular.

For any singular $H\in SP$, assume that $Hu=u$ where $0\neq u\in\R^{2d}$. From Property \ref{pro:part_symm} and Property \ref{pro:part_SP}, we know that $u$ is a part of a symplectic matrix $Q=[u\ \star]$. Therefore $Q^{-1}HQe_{1}=Q^{-1}Hu=Q^{-1}u=e_{1}$, which means that $Q^{-1}HQ\in\mathcal{W}_{e_{1}}=\mathcal{V}_{e_{1}}$ due to Theorem \ref{thm:SPS_thm}. We immediately know that $H\in SPS$.
\end{proof}

\subsection{Unconstrained optimization}
An optimization problem with symplectic constraint is in the following form
\begin{equation}
\label{eq:ConOP}
\min\limits_{X \in \R^{2d \times 2d}}f(X),\quad s.t.\ X^TJX=J.
\end{equation}
Because of the antisymmetry of $X^TJX$, the constraint of the problem (\ref{eq:ConOP}) is made up of $2d^2-d$ equations. There have been many works on optimization on the real symplectic group \cite{birtea2020optimization,fiori2011solving,wang2018riemannian}, in which one performs optimization by considering the gradients along the manifold. In this work, the unit triangular factorization provides an approach to the symplectic optimization from a new perspective, i.e., optimizing 
in a higher dimensional unconstrained parameter space. By (\ref{eq:SP_para}), problem (\ref{eq:ConOP}) is equivalent to
\begin{equation*}
\min\limits_{Pa(S_{i})\in\R^{\frac{d(d+1)}{2}}}f(H(Pa(S_{1}),\cdots,Pa(S_{9}))),
\end{equation*}
which is indeed an unconstrained optimization problem. In recent years, with the expansion of variables required by practical application, the limitations of existing constrained optimization methods are reflected in practice. Fortunately, the unconstrained parametrization of the symplectic matrices allows us to circumvent this limitation and apply faster and more efficient unconstrained optimization algorithms, such as the recently popular deep learning techniques, which are able to deal with the optimization problems with billions of unconstrained parameters. In fact, this method has been utilized in our recent work \cite{jin2020sympnets}, where we construct symplectic neural networks based on the factorization, and achieve a great success. In such a case, the unit triangular factorization-based optimization can be implemented directly within the deep learning framework and performs well, while the traditional Riemannian-steepest-descent approach faces challenges.
\section{Conclusions}
The factorization theorems shown in this work overcomes some defects of conventional factorizations. \cite{bunse1986matrix,dopico2009parametrization,higham2006symmetric,sadkane2009note,xu2003svd} provide several factorizations of the matrix symplectic group. However, all the factorizations require cells of symplectic interchanges, permutation matrices or symplectic-orthogonal matrices, which are not elementary enough hence hard to be freely parameterized. \cite{flaschka1991analysis,mackey2003determinant} factor the $2d$-by-$2d$ symplectic matrices as the products of at most $4d$ symplectic transvections, which can freely parameterize the matrix symplectic group, nevertheless, may be unstable with a large $d$ in practice.

Theorem \ref{thm:fac_thm} proves that any symplectic matrix can be factored into no more than 9 unit triangular symplectic matrices. The core of the proof is Theorem \ref{thm:nonsin_thm}, which writes a symplectic matrix as the product of a symplectic matrix with nonsingular left upper block and a unit triangular symplectic matrix. Subsequently, we derive some important corollaries by Theorem \ref{thm:fac_thm} and Theorem \ref{thm:nonsin_thm}, such as, (\romannumeral1) the determinant of symplectic matrix is one, (\romannumeral2) the matrix symplectic group is path connected, (\romannumeral3) all the unit triangular symplectic matrices form a set of generators of the matrix symplectic group. Furthermore, this factorization yields effective methods for unconstrained parametrization of the matrix symplectic group and its structured subsets. With the unconstrained parametrization, we are able to apply the unconstrained optimization algorithms to the problems with symplectic constraints.

Although we proved that any symplectic matrix can be factored into 9 unit triangular symplectic matrices, it is still unknown that whether ``9'' is the optimal number in this factorization. What we know regarding to the optimal number is that it is indeed between 4 and 9. This problem is left for future work.

\section*{Acknowledgments}
This research is supported by the Major Project on New Generation of Artificial Intelligence from MOST of China (Grant No. 2018AAA0101002), and National Natural Science Foundation of China (Grant No. 11771438).

\bibliographystyle{abbrv}
\bibliography{references}

\end{document}